\documentclass[11pt]{article}
\usepackage{amsmath, amscd, amssymb, latexsym, epsfig, color, amsthm,enumerate, mathtools, tikz-cd}
\usepackage[all]{xy}
\usepackage{verbatim}
\usepackage{esvect}
\numberwithin{equation}{section}
\def\proof{\smallskip\noindent {\it Proof: \ }}

\newtheorem{theorem}{Theorem}[section]
\newtheorem{proposition}[theorem]{Proposition}

\newtheorem{conjecture}[theorem]{Conjecture}
\newtheorem{lemma}[theorem]{Lemma}

\newtheorem{problem}[theorem]{Problem}

\theoremstyle{definition}
\newtheorem{definition}[theorem]{Definition}

\DeclareMathOperator{\cs}{cs}
\DeclareMathOperator{\sn}{sn}

\newcommand{\N}{{\mathbb N}}
\newcommand{\R}{{\mathbb R}}

\newcommand{\A}{{\mathcal A}}

\newcommand{\C}{{\mathcal C}}

\newcommand{\T}{{\mathcal T}}

\title{Neighborly spheres and transversal numbers}
\author{
	Isabella Novik\thanks{Research of IN is partially\textsl{} supported by NSF grant DMS-1953815 and by Robert R.~\&  Elaine F.~Phelps Professorship in Mathematics. }\\
	\small Department of Mathematics\\[-0.8ex]
	\small University of Washington\\[-0.8ex]
	\small Seattle, WA 98195-4350, USA\\[-0.8ex]
	\small \texttt{novik@uw.edu}
	\and 
	Hailun Zheng\thanks{Research of HZ is partially supported by a postdoctoral fellowship from ERC grant 716424 - CASe.}\\
	\small Department of Mathematical Sciences\\[-0.8ex]
	\small University of Copenhagen\\[-0.8ex]
	\small Universitesparken 5, 2100 Copenhagen, Denmark \\[-0.8ex]
	\small \texttt{hz@math.ku.dk}
}
\begin{document}
\maketitle
	\section{Introduction}
	The goal of this paper is to advertise several old and new problems related to the number of simplicial spheres (Section 3), the number of neighborly simplicial spheres (Section 4), and the number of centrally symmetric simplicial spheres that are cs-neighborly (Section 5). Neighborly spheres provide maximizers for the upper bound type problems. They are also natural candidates for having large transversal numbers. Thus, in the second part of the paper (Section 6), we discuss problems as well as some results aimed at understanding transversal numbers of  simplicial spheres. In particular, we estimate transversal ratios of some of the neighborly spheres presented in Sections 4 and 5 and conclude that for $k\geq 2$, there exist families of simplicial $2k$-spheres such that their transversal ratios are at least $2/5-o(1)$ as the number of vertices tends to infinity. (The previously known bound was $1/(k+1)$). 
	
	The tools developed and the results proved in the last decade suggest that significant progress on some of these problems might be just around the corner. In short, we believe that these problems deserve to be looked at right now!
	
	\section{Basic definitions}
	We assume that the reader is familiar with basics of convex polytopes and simplicial complexes. However, for completeness sake,  below we review some of the definitions, especially those related to simplicial complexes. For all undefined terminology related to polytopes, the reader is referred to Ziegler's book \cite{Ziegler}.
	
	In this paper we only consider finite simplicial complexes.  A {\em simplicial complex} $\Delta$ on a finite vertex set $V$ is a collection of subsets of $V$ that is closed under inclusion, i.e., if $F\in\Delta$ and $G\subset F$, then $G\in\Delta$. We also usually assume that every singleton is an element of $\Delta$, that is, $\{v\}\in\Delta$ for all $v\in V$. The elements of $V$ are called {\em vertices} while the elements of $\Delta$ are called {\em faces}. For a vertex $v$, we abuse notation and write $v\in\Delta$ instead of $\{v\}\in\Delta$.
	
	We say that $F\in\Delta$ is an {\em $i$-face} or an {\em $i$-dimensional face} if $|F|=i+1$. The {\em dimension} of $\Delta$ is defined as $\dim\Delta \coloneqq \max\{\dim F : F\in\Delta\}$. The number of $i$-faces is denoted by $f_i(\Delta)$. We refer to $0$-faces as {\em vertices}, $1$-faces as {\em edges}, and the maximal under inclusion faces as {\em facets}. We say that $\Delta$ is {\em pure} if all facets of $\Delta$ have the same dimension. For instance, the complex whose faces are the empty set, the vertices $v_1,\ldots,v_n$, and the edges $\{v_i,v_{i+1}\}$ for all $1\leq i\leq n-1$, is a pure $1$-dimensional complex known as a {\em path}. The notation we adopt for this complex is $(v_1,v_2,\ldots,v_n)$.
	
	Two examples of simplicial complexes worth mentioning are the {\em simplex on $V$}, denoted $\overline{V}$, and the {\em boundary complex} of the simplex $\overline{V}$, denoted $\partial \overline{V}$. The former complex consists of {\em all} subsets of $V$ while the latter complex consists of all subsets of $V$ but $V$ itself. 
	
	An important operation on simplicial complexes is that of joins: If $\Delta$ and $\Gamma$ are simplicial complexes on disjoint vertex sets $V$ and $V'$, then the \textit{join} of $\Delta$ and $\Gamma$ is the simplicial complex $\Delta * \Gamma = \{\sigma \cup \tau \ : \ \sigma \in \Delta \text{ and } \tau \in \Gamma\}$ with vertex set $V\cup V'$. A special case is the \emph{cone} over $\Delta$ with apex $v$, $\Delta * v$, defined as  the join of $\Delta$ and the simplex $\overline{\{v\}}$.
	
	A {\em polytope} is the convex hull of finitely many points in a Euclidean space. One example is given by a (geometric) simplex: the convex hull of affinely independent points. A {\em face} of a polytope $P$ is the intersection of $P$ with any supporting hyperplane. A polytope $P$ is {\em simplicial} if all facets of $P$ are simplices.
	
	There is a natural way to associate with every simplicial complex $\Delta$ its {\em geometric realization}, $\|\Delta\|$: it is built out of geometric simplices in a way that every two simplices intersect along a common (possibly empty) face and the collection of vertex sets of faces of $\|\Delta\|$ is $\Delta$. We say that $\Delta$ is a {\em simplicial  $(d-1)$-sphere} if $\|\Delta\|$ is homeomorphic to a $(d-1)$-sphere; we also say that $\Delta$ is a simplicial $d$-ball if $\|\Delta\|$ is homeomorphic to a $d$-ball. For instance, $\overline{V}$ is a simplicial ball, while $\partial\overline{V}$ as well as the boundary complex of any simplicial polytope is a simplicial sphere. 
	
	Many balls and spheres we discuss in this paper are {\em PL balls} and {\em PL spheres}. A PL $d$-ball is a simplicial complex PL homeomorphic to a (geometric) $d$-simplex. Similarly, a PL $ (d-1)$-sphere is a simplicial complex PL homeomorphic to  the boundary complex of a $d$-simplex.
	
	A simplicial $(d-1)$-sphere $\Delta$ is called $k$-neighborly if every $k$ vertices of $\Delta$ form the vertex set of a face. The boundary complex of a $d$-simplex is $d$-neighborly. On the other hand, it is not hard to see (for instance, from Dehn-Sommerville relations \cite{Klee64,Sommerville}) that a simplicial $(d-1)$-sphere with at least $d+2$ vertices cannot be more than $\lfloor d/2\rfloor$-neighborly. Furthermore,  $\lfloor d/2\rfloor$-neighborly simplicial $(d-1)$-spheres with arbitrarily many vertices do exist. One famous example is given by the boundary complex of the cyclic polytope whose construction is outlined below. 
	
	The {\em moment curve} $M=M_d: \R\to \R^d$ is defined by $M(t)\coloneqq(t,t^2,\ldots,t^d)$. Let $t_1<t_2<\cdots<t_n$ be any $n\geq d+1$ distinct real numbers. The  {\em cyclic polytope}, $C(d,n)$ is the convex hull of the points $M(t_1),\ldots, M(t_n)$. The cyclic polytope has several amazing properties \cite{Gale63, Ziegler}: it is a simplicial $d$-polytope with $n$ vertices; its combinatorial type is independent of the choice of $t_1,\ldots,t_n$; and it is $\lfloor d/2\rfloor$-neighborly. That the combinatorial type of $C(d,n)$ is independent of the choice of $t_1,\ldots,t_n$ is a consequence of a beautiful result by Gale \cite{Gale63} that provides a complete characterization of the facets of $C(d,n)$; it is known as the {\em Gale evenness condition}. 
	
	The interest in $\lfloor d/2\rfloor$-neighborly simplicial $(d-1)$-spheres comes in part from the celebrated Upper Bound Theorem due to Stanley \cite{Stanley75}. It asserts that in the class of all simplicial $(d-1)$-spheres with $n$ vertices, any $\lfloor d/2\rfloor$-neighborly sphere (e.g., the boundary complex of the cyclic polytope) simultaneously maximizes all the face numbers.
		
	Another polytope we will encounter in this paper is the {\em cross-polytope} $\C^*_d$ defined as the convex hull of $\{\pm e_1,\pm e_2,\ldots,\pm e_d\}$, where $e_1,\ldots,e_d$ are the endpoints of the standard basis in $\R^d$. As an abstract simplicial complex, the boundary complex $\partial \C^*_d$ of $\C^*_d$ is the collection of all subsets of $\{\pm 1,\ldots,\pm d\}$ that contain at most one vertex from each pair $\{\pm i\}$.
	
	\section{How many simplicial spheres are there?}
    
	Denote by $s(d,n)$ the number of simplicial $(d-1)$-spheres with $n$ {\em labeled} vertices. How large is $s(d, n)$? 
		
	We first discuss the case of $d=3$. It is well known that $s(3,n)=2^{{\Theta(n\log n)}}$. In fact, it is even known that the number of {\em combinatorial types} of simplicial $2$-spheres with $n$ vertices is asymptotically exponential in $n$. This follows from works of Tutte \cite{Tutte62, Tutte80}, Brown \cite{Brown}, and Richmond and Wormland \cite{RichmWorm82} along with Steinitz' theorem \cite[Chapter 4]{Ziegler} which implies that all simplicial (and even polyhedral) $2$-spheres are boundary complexes of $3$-polytopes.

	Assume now that $d\geq 4$.  The currently best known bounds are:
	\begin{theorem} \label{thm:s(d,n)}
		For $d\geq 4$, 
		\[
		2^{\Omega(n^{\lfloor d/2\rfloor})} \leq s(d,n) \leq 2^{O(n^{\lfloor d/2\rfloor}\log n)}.
		\]
	\end{theorem}
	
	\noindent For comparison, we should mention that there are much fewer polytopes. By works of Goodman and Pollack \cite{GoodmanPollack} and Alon \cite{Alon-86}, for $d\geq 4$, the number of $d$-polytopes (simplicial and non-simplicial) with $n$ labeled vertices is $2^{\Theta(n\log n)}$. The number of polytopes with unlabeled vertices satisfies the same asymptotics.
	
	Let's take a closer look at Theorem \ref{thm:s(d,n)}.\footnote{Note that although we count the number of spheres with {\em labeled} vertices (i.e., we count complexes up to equality), since $n! = 2^{O(n\log n)}$, the bounds of Theorem \ref{thm:s(d,n)} also apply to the number of combinatorial types of simplicial $(d-1)$-spheres on $n$ vertices, i.e., to the number of spheres counted up to isomorphism.} The upper bound on $s(d,n)$ essentially follows from Stanley's Upper Bound Theorem, see \cite[Section 4.2]{Kal}. The lower bound for odd values of $d$ is due to Kalai \cite{Kal}. In fact, Kalai proved that for {\em all} $d\geq 5$ (even and odd), $s(d,n)\geq 2^{\Omega(n^{\lfloor \frac{d-1}{2}\rfloor})}$. For $d=4$, Pfeifle and Ziegler \cite{PfeiZieg} showed that $s(4,n)\geq 2^{\Omega(n^{5/4})}$. The lower bound of Theorem \ref{thm:s(d,n)} for even values of $d$ is due to Nevo, Santos, and Wilson \cite{NeSanWil}. As the reader will notice, the exponents in the lower and upper bounds differ by the factor of $\log n$. This brings us to 
	
	\begin{problem} \label{prob:s(d,n)}
		Find tighter bounds on $\log s(n,d)$. Is $\log s(d,n)=\Theta(n^{\lfloor d/2\rfloor})$ or is it closer to $\Theta(n^{\lfloor d/2\rfloor}\log n))$?
	\end{problem}
	
	The lower bounds in Theorem \ref{thm:s(d,n)} are obtained by ingenious constructions. Kalai's construction of many $(2k-2)$-spheres \cite{Kal}  (for $k\geq 3$) is based on squeezed balls --- certain full-dimensional subcomplexes of the boundary complex of the cyclic polytope $C(2k,n)$. The boundary complexes of squeezed balls, called {\em squeezed spheres}, provide us with $2^{\Omega(n^{k-1})}$ simplicial $(2k-2)$-spheres with $n$ vertices. One of constructions by Nevo, Santos, and Wilson  \cite{NeSanWil} (at least in the case of $d=4$) is also based on the cyclic polytope. Furthermore, all the constructions in \cite{Kal, NeSanWil, PfeiZieg} are PL spheres, and, as was proved by Lee \cite{Lee00}, all of Kalai's squeezed spheres are even shellable. Since there are many non-shellable (and even non-constructible) simplicial spheres \cite{HachimoriZiegler}, there seems to be large unexplored terrain for potential additional constructions. One may even speculate that the value of $s(d,n)$ is closer to the upper bound in Problem \ref{prob:s(d,n)}. To this end, it is worth noting that there are constructions of simplicial $(d-1)$-spheres, $d\geq 6$, that are {\em not PL} and have only $d+12$ vertices \cite{BjornerLutz}.
	
	A related problem, where much less is known, is 
	
	\begin{problem} \label{prob:spheres-facets}
		For $d\geq 4$, how many combinatorial types of simplicial $(d-1)$-spheres with $N$ facets are there? Is the number of such spheres at most exponential in $N$, that is, is it bounded from above by $(C_d)^N$, where $C_d$ is some constant that depends only on $d$?
	\end{problem}
	
	This problem, for $d=4$, was raised by Ambj\o{}rn, Durhuus, and Jonsson \cite{ADJ91} and advertised by Gromov \cite[pp.~156--157]{Gromov2010}. For $2$-spheres, the answer to the second question of Problem \ref{prob:spheres-facets} is indeed affirmative, see \cite{Tutte62, Tutte80}.  For $d\geq 4$, Benedetti and Ziegler \cite{BZ} proved that the number of shellable (and, in fact, even the number of locally constructible) $(d-1)$-spheres with $N$ facets is at most exponential in $N$. Some additional results on exponential growth were recently established by Adiprasito and Benedetti \cite{ABcheeger}. However, the general case of Problem \ref{prob:spheres-facets} remains wide open. It is also worth mentioning  that if $\Delta$ is a simplicial $(d-1)$-sphere with $n$ vertices, then  $f_{d-1}(\Delta) \leq O(n^{\lfloor d/2\rfloor})$. (This is a simple consequence of the Upper Bound Theorem \cite{Stanley75}.) Thus, an affirmative answer to Problem  \ref{prob:spheres-facets} would imply that 
	\[s(d,n) \leq n! \cdot \sum_{N=1}^{O(n^{\lfloor d/2\rfloor})} (C_d)^N =
	2^{O(n^{\lfloor d/2\rfloor})}, \quad \mbox{and hence, that } s(d,n)=2^{\Theta(n^{\lfloor d/2\rfloor})}.\]

	\section{How many neighborly simplicial spheres are there?}  \label{section:neighb-spheres}
	Since $\lfloor d/2\rfloor$-neighborly simplicial $d$-polytopes and $(d-1)$-spheres serve as maximizers in the Upper Bound Theorem \cite{McMullen70,Stanley75}, one might be tempted to think that the property of being $\lfloor d/2\rfloor$-neighborly is quite rare. This, however, is not the case: Shemer \cite{Shemer} introduced a sewing construction and used it to produce many neighborly polytopes. The combinatorial part of the sewing construction can be roughly described as follows: one starts with a neighborly PL sphere $\Delta$ with $n$ vertices, finds in $\Delta$ a full-dimensional PL ball $B$ that satisfies certain nice properties, and then produces a new neighborly sphere $\Delta'$ with $n+1$ vertices  by replacing the ball $B$ with another ball --- the cone over the boundary of $B$, where the cone vertex is a new vertex.\footnote{Since Shemer's goal was to produce neighborly {\em polytopes} rather than spheres, he also needed to ensure that this ball $B$ can be chosen in a way that allows to place the new vertex beyond all facets in $B$ and beneath the rest of facets.}

	Padrol \cite{Padrol-13} extended and generalized Shemer's technique to produce even more neighborly polytopes: he constructed on the order of $n^{dn/2}$ polytopes of dimension $d$ with $n$ (labeled) vertices that are $\lfloor d/2\rfloor$-neighborly. Until very recently, this lower bound was the best lower bound not only on the number of neighborly polytopes, but also on the total number of $d$-polytopes with $n$ (labeled) vertices. The current best lower bound on the total number of polytopes is on the order of $n^{dn}$, see \cite{PPS}.
	
	Let $\sn(d,n)$ be the number of $\lfloor d/2\rfloor$-neighborly simplicial $(d-1)$-spheres with $n$ {\em labeled} vertices.	Motivated by Shemer's results, Kalai \cite{Kal} proposed the following bold conjecture:
	\begin{conjecture} \label{conj:many-neighb-spheres}
		For all $d\geq 4$, $\lim_{n\to\infty} (\log \sn(d,n)/ \log s(d,n))=1$.
	\end{conjecture} 
	
	The current best lower bound on $\sn(d,n)$ is due to Novik and Zheng \cite{NZ-neighborly} who proved that for $d\geq 5$,
	\[
	\sn(d,n) \geq 2^{\Omega(n^{\lfloor (d-1)/2\rfloor})}.
	\]
	The construction of \cite{NZ-neighborly} relies on Kalai's squeezed balls \cite{Kal}. More precisely, it relies on  {\em differences} of appropriately chosen squeezed balls, called {\em relative squeezed balls}, see \cite[Theorem 3.1(1)]{NZ-neighborly} and Section 6.1 below for precise conditions. Such balls are $(2k-1)$-dimensional subcomplexes of  $\partial C(2k,n)$ --- the boundary complex of $C(2k,n)$. The desired $(k-1)$-neighborly $(2k-2)$-spheres are obtained as the boundaries of relative squeezed balls while the desired $k$-neighborly $(2k-1)$-spheres are obtained by the sewing construction where a relative squeezed ball $B$ is deleted from  $\partial C(2k,n)$ and replaced by the cone over the boundary of $B$.
	
	While Conjecture \ref{conj:many-neighb-spheres} is likely out of reach at the moment, in view of the result of Nevo, Santos, and Wilson  \cite{NeSanWil}, the following problem might be more accessible:
	
	\begin{problem}
	Let $d\geq 4$. Are there $2^{\Omega(n^{\lfloor d/2\rfloor})}$ simplicial $(d-1)$-spheres on $n$ vertices that are $\lfloor d/2\rfloor$-neighborly? 
	\end{problem}
	
	\section{How many centrally symmetric spheres are there?}  \label{section:cs-spheres}
	A polytope $P\subset \R^d$ is {\em centrally symmetric} (or {\em cs} for short) if $P=-P$. For instance, the cross-polytope is cs. Similarly, a simplicial sphere $\Delta$ is {\em centrally symmetric} (or {\em cs}) if the vertex set of $\Delta$ is endowed with a {\em free involution} $\alpha$ that induces a free involution on the set of all nonempty faces of $\Delta$.  In more detail, for every nonempty face $F\in \Delta$, the following holds:
	\[ \alpha(F)\in \Delta, \quad \alpha(F)\neq F, \quad \mbox{and } \alpha(\alpha(F))=F.
	\]
	We call vertices $v$ and $\alpha(v)$ {\em antipodal}; for brevity we write $-v$ instead of $\alpha( v)$.
	
	Note that if $\Delta$ is cs, then $v$ and $-v$ never form an edge (as otherwise such an edge would be fixed by the involution). This leads to an adjusted notion of neighborliness for cs simplicial spheres: a cs simplicial sphere $\Delta$ is called {\em cs-$k$-neighborly} if every subset of $k$ vertices of $\Delta$ that does not contain a pair of antipodal vertices is the vertex set of a face.
	
	While cs-neighborliness of cs polytopes is very restricted (for example, a cs $d$-polytope with $2^d$ or more vertices cannot be even cs-$2$-neighborly \cite{LinNov, Novik-cs2neighb}), cs $(d-1)$-spheres that are cs-$\lfloor d/2\rfloor$-neighborly and have an arbitrary large number of vertices do exist. The following result was proved by Jockusch \cite{Jockusch95} for $d=4$ and by Novik and Zheng \cite{NZ-cs-neighborly} for all $d$:
	
	\begin{theorem}  \label{thm:cs-neighb}
		For all values of $d\geq 4$ and $n\geq d$, there exists a cs PL $(d-1)$-sphere with $2n$ vertices, $\Delta^{d-1}_n$, that is cs-$\lfloor d/2\rfloor$-neighborly.
	\end{theorem} 
	
	The construction uses the following cs analog of combinatorial sewing. It starts with the boundary complex of the cross-polytope: $\Delta^{d-1}_{d}\coloneqq\partial \C^*_d$. At each step of the construction, it is shown that there exists a pair of antipodal $(d-1)$-balls $\pm B$ in $\Delta^{d-1}_{n-1}$ that share no common facets and have other nice properties. The complex $\Delta^{d-1}_{n}$ is then obtained from $\Delta^{d-1}_{n-1}$ by introducing two new vertices, $n$ and $-n$,  and replacing $B$ and $-B$ with the cones over the boundary complexes of $B$ and $-B$: $\partial B * n$ and $\partial(-B)*(-n)$, respectively; see Definition \ref{main-constr} below for more details.

	The significance of Theorem \ref{thm:cs-neighb} is that exactly as in the case of spheres without symmetry assumption, in the class of cs simplicial $(d-1)$-spheres with $2n$ vertices, a cs-$\lfloor d/2\rfloor$-neighborly sphere simultaneously maximizes all the face numbers. This follows from work of Adin \cite{Adin93} and Stanley (unpublished). Also, as in the case without symmetry assumption, it is very natural to consider the following problems:
	
	\begin{problem} \label{prob:cs(d,n)}
		Let $d\geq 4$. Denote by $\cs(d, n)$ the number of cs simplicial $(d-1)$-spheres with $2n$ labeled vertices. How large is $\cs(d, n)$?
	\end{problem}

	\begin{problem} \label{prob:ncs(d,n)}
		Let $d\geq 4$. How many cs simplicial $(d-1)$-spheres with $2n$ vertices are cs-$\lfloor d/2\rfloor$-neighborly? Can one construct  $2^{\Omega(n^{\lfloor (d-1)/2\rfloor})}$ such spheres, or, more optimistically, $2^{\Omega(n^{\lfloor d/2\rfloor})}$ such spheres?
	\end{problem}
	
	An additional impetus for Problem \ref{prob:ncs(d,n)} comes from a result of McMullen--Shephard \cite{McMShep} asserting that for $d\geq 4$, a cs $(d-1)$ sphere that is cs-$\lfloor d/2\rfloor$-neighborly and has more than $2d+2$ vertices cannot be realized as the boundary of cs $d$-polytope. In contrast with the situation for spheres without symmetry assumption, not much is known about this problem: for $d=4$, there are at least $\Omega(2^n)$ pairwise {\em non-isomorphic} cs spheres with $2n$ vertices that are cs-$2$-neighborly, see \cite[Section 7]{NZ-cs-neighborly-new}. For $d=2k>4$ and $n\gg 0$, only two non-isomorphic constructions are available at present (see  \cite[Thms.~5.5 and 5.7]{NZ-cs-neighborly-new}). One of these two spheres is $\Delta^{2k-1}_n$. The other one, which we denote by $\Lambda^{2k-1}_n$, is the link of one of the edges in $\Delta^{2k+1}_{n+2}$. 

As for Problem \ref{prob:cs(d,n)}, we claim that $\cs(d, n)$ satisfies the inequalities of Theorem \ref{thm:s(d,n)}, that is, 
    $$2^{\Omega(n^{\lfloor d/2\rfloor})} \leq \cs(d,n) \leq 2^{O(n^{\lfloor d/2\rfloor}\log n)}.$$
    
    We provide a sketch of the proof. The upper bound follows from the fact that $\cs(d, n)\leq s(d, n)$. It remains to prove the lower bound for $d=2k$. (For $d=2k+1$, simply count suspensions of cs $(2k-1)$-spheres.) Assume that $n=mk$. By \cite[Proposition 4.3]{NZ-cs-neighborly-new}, a cs simplicial $(2k-1)$-sphere $\Lambda^{2k-1}_{2n}$ from the family of cs-$k$-neighborly spheres mentioned above satisfies the following property: it contains an isomorphic image of the ball that is the join of $k$ paths $(1,2,\dots,m)$, $(m+1,m+2,\dots,2m)$, $\dots$, and $((k-1)m+1,..., km)$. Denote this isomorphic image by $I$. By symmetry, $\Lambda^{2k-1}_{2n}$ also contains $-I$. We now apply Construction 3 from \cite{NeSanWil} to $I$, that is, we retriangulate $I$ in such a way that the resulting ball has $O(n)$ labeled vertices and its boundary is the same as $\partial I$. Let $\T$ be the set of new balls. As shown in \cite{NeSanWil}, the size of $\T$ is $2^{\Omega(n^{\lfloor d/2\rfloor})}$.

	Consider the prism $\partial I \times [0,1]$. Its facets are prisms over simplices. Triangulate this polyhedral complex without introducing new vertices. For each $B\in\T$, let $\Gamma(B)$ be this triangulation of $\partial I \times [0,1]$ with a copy of $B$ glued to the top, i.e., to $\partial I \times \{1\}=\partial B \times \{1\}$, along the boundary of $B$. Then $\Gamma(B)$ is a PL ball whose boundary is $\partial I$ (the bottom of the prism). The sphere obtained from $\Lambda^{2k-1}_{2n}$ by replacing $I$ with $\Gamma(B)$ and $-I$ with $-\Gamma(B)$ is then a cs simplicial $(2k-1)$-sphere with $O(n)$ vertices. In this way, we obtain $2^{\Omega(n^{\lfloor d/2\rfloor})}$ cs $(2k-1)$-spheres with $O(n)$ vertices.

	\section{The transversal numbers of simplicial spheres}
    Let $H=(V, E)$ be a hypergraph. A {\em transversal} of $H$ is a subset $T$ of $V$ that intersects all edges of $H$. The {\em transversal number} $\tau(H)$ of $H$ is the minimum cardinality of transversals of $H$. Transversal numbers of interesting families of graphs and hypergraphs have been studied extensively, see, for example, \cite{Alon90,Alonetall02}.

    In this section, we restrict our discussion to transversal numbers of hypergraphs that arise from simplicial spheres. Given a simplicial sphere $\Delta$, we define a hypergraph $H(\Delta)=(V, E)$, where $V$ is the vertex set of $\Delta$, and $E$ is the set of facets of $\Delta$. Then $\tau(H(\Delta))$ is the minimum number of vertices needed so that the union of their (closed) vertex stars is the entire complex $\Delta$, i.e., this union does not miss any facet of $\Delta$.
		In addition to the transversal number $\tau(H(\Delta))$, we are also interested in the {\em transversal ratio} $\mu(\Delta)\coloneqq\tau(H(\Delta))/f_0(\Delta)$. 
    \begin{problem}  \label{prob:mu_d}
    	Let $d\geq 4$. How does the function $$\mu_d(n)\coloneqq\max\{\mu(\Delta): \Delta\, \text{is a simplicial $(d-1)$-sphere with $n$ vertices}\}$$ behave when $n$ is large?
			Find the value of  $\mu_d\coloneqq\limsup_{n\to\infty} \mu_d(n).$
    \end{problem}
    
		 In \cite{BriggsDobbinsLee}, it is shown that $\mu_3=1/2$: $\mu_3\leq 1/2$ follows from the four color theorem and $\mu_3\geq 1/2$ follows from the existence of an infinite family of $2$-spheres with transversal ratio $1/2$. The problem remains wide open for $d\geq 4$. Several lower bounds on $\mu_d$ proved in \cite{BriggsDobbinsLee} can be summarized as follows: $\mu_4\geq 11/21$, while for $k\geq 3$, $\mu_{2k}\geq 1/2$ and $\mu_{2k-1}\geq 1/k$. The proof of the bound $\mu_4\geq 11/21$ begins with a construction of a $2$-neighborly PL $3$-sphere with $21$ vertices whose transversal number is $11$; it is found with computer help. The bound $\mu_{2k}\geq 1/2$ is obtained by considering the family of the boundary complexes of the cyclic polytopes $\{\partial C(2k,n) : n\geq 2k+1\}$. This approach does not work for even-dimensional spheres as the transversal number of $\partial C(2k-1,n)$ is equal to two independently of $n$. The bound $\mu_{2k-1}\geq 1/k$ is verified by looking at the family of the boundary complexes of stacked $(2k-1)$-polytopes. This discussion leads to the following weaker version of Problem \ref{prob:mu_d}:
		
		\begin{problem}
		For a fixed $d\geq 4$, is $\mu_d$ is bounded away from $1$? Is $\limsup_{d\to\infty} \mu_d$ strictly smaller than $1$ or is it equal to $1$?
		\end{problem}

    A simplicial sphere is {\em flag} if it is the clique complex of its graph. Under the metric that assigns the same length $\pi/2$ to all edges of the sphere, the flag spheres have particularly nice geometric properties; for example, the vertex stars are geodesically convex, see \cite{Gromov87}. This motivates our final problem.

    \begin{problem} \label{prob:flag}
	Let $d\geq 3$. Find the value of  $\lambda_d\coloneqq\limsup_{n\to\infty} \lambda_d(n)$, where $$\lambda_d(n)\coloneqq\max\{\mu(\Delta) : \Delta\, \text{is a flag $(d-1)$-sphere with $n$ vertices}\}.$$ Is $\lambda_d<\mu_d$?  
    \end{problem}
    
In the rest of the paper, we obtain some estimates on transversal ratios. Our main result is that for all odd $d\geq 5$, $\mu_d\geq 2/5$. It is curious to note that most of the large transversal ratios mentioned above are attained by neighborly spheres. This suggests we consider various families of neighborly spheres to find better lower bounds on $\mu_d$. To start, we notice that $\mu(\Delta)\geq 1/4$ for {\em any} 2-neighborly 3-sphere $\Delta$ with $n$ vertices: by the 2-neighborliness of $\Delta$ and the Dehn--Sommerville relations, $f_3(\Delta)=f_1(\Delta)-f_0(\Delta)=n(n-3)/2$. Now, the link of each vertex is a $2$-sphere with $n-1$ vertices. Thus, each vertex belongs to $2(n-1)-4=2(n-3)$ facets. It follows that $\mu(\Delta)\geq f_3(\Delta)/2(n-3)n=1/4.$ (In fact, by using the inclusion-exclusion principle, one can even improve this lower bound to $\frac{3}{2}-\sqrt{\frac{3}{2}}$; we invite the reader to check this fact.)

As promised, we now turn to proving that $\mu_d\geq 2/5$ for all odd $d\geq 5$. We achieve this by estimating the transversal ratios of two classes of highly neighborly spheres that we encountered in Sections \ref{section:neighb-spheres} and \ref{section:cs-spheres}, respectively: the family of relative squeezed spheres and the family of cs neighborly spheres $\{\Delta^{d-1}_n : n\geq d\}$.
    
    \subsection{The traversal numbers of relative squeezed spheres}
	We start by reviewing definitions of squeezed balls and relative squeezed balls; the reader is referred to \cite{Kal,NZ-neighborly} for more details.
		
    Let $k\geq 1$, and for $1\leq m \leq n$, let $[m,n]$ denote the set $\{m,m+1,\ldots,n\}$.
	The poset $\mathcal{F}_{2k}^{[m,n]}$ is defined as follows: as a set, it consists of the following $2k$-sets (which are in fact facets of $\partial C(2k, n)$): 
    $$\{\{i_1, i_1+1, i_2, i_2+1, \dots, i_k, i_k+1: m\leq i_1, \,i_k+1\leq n,\, \mbox{ and } i_j\leq i_{j+1}-2 \,\,\,  \forall 1\leq j\leq k-1\}\},$$ ordered by the standard partial order $\leq_p$ (the product order on $\N^{2k}$). Given an antichain $S$ in $\mathcal{F}_{2k}^{[1,n]}$, one can define $B(S)$ to be the pure simplicial complex whose facets are the elements of the order ideal of $\mathcal{F}_{2k}^{[1,n]}$ generated by $S$. It is a theorem of Kalai \cite{Kal} that for any antichain $S$, $B(S)$ is a PL $(2k-1)$-ball, called a {\em squeezed ball}; its boundary, $\partial B(S)$, is then a PL $(2k-2)$-sphere, called a {\em squeezed sphere}.
		
 Given an antichain $S$ in $\mathcal{F}_{2k}^{[1,n]}$, one can consider another (potentially empty) antichain
\[
S-\mathbf{1}_{2k}\coloneqq\{\{x_1-1, x_1, x_2-1, x_2, \dots, x_k-1, x_k\} : \{x_1, x_1+1, x_2, x_2+1, \dots, x_k, x_k+1\}\in S, \;x_1>1\}.
\]
Let $B_S$ denote the pure simplicial complex whose facets are all facets of $B(S)$ that are not facets of $B(S-\mathbf{1}_{2k})$. It is proved in \cite{NZ-neighborly} that for any antichain $S$, $B_S$ is a PL $(2k-1)$-ball, and hence, $\partial B_S$ is a PL $(2k-2)$-sphere. We call $B_S$ a {\em relative squeezed ball} and $\partial B_S$ a {\em relative squeezed sphere}. An important property of these complexes is that if $S$ contains the element $[1,2]\cup [n-2k+3,n]$, then $\partial B_S$ is a $(k-1)$-neighborly sphere with vertex set $[1,n]$, see \cite[Theorem 3.1]{NZ-neighborly}.
 
For the rest of this subsection we fix $k\geq 3$ and we let $A_n$ denote the following antichain in $\mathcal{F}_{2k}^{[1,n]}$; to simplify notation we write $A$ instead of $A_n$ when $n$ is fixed or understood from context:
    $$A=A_n\coloneqq\{F_1, F_2, \dots, F_{\lfloor \frac{n}{2}\rfloor -k+1}\}, \qquad \mbox{where }F_i=[i, i+1]\cup[n-2k+4-i, n-i+1].$$ In particular, $F_1=[1,2]\cup [n-2k+3,n]$, and so $\partial B_{A_n}$ is $(k-1)$-neighborly. Our goal is to estimate the transversal ratios of $\partial B_{A_n}$ as $n\to\infty$, see Proposition \ref{prop:mu-of-partialB_S}. To do so, it is useful to describe some of the facets of 
$\partial B_A=\partial B_{A_n}$ (for a fixed $n$). This is done in the following lemma.

    \begin{lemma}\label{lm: relative squeezed sphere facets}
    Let $k\geq 3$ and let $n\geq 2k+1$. The complex $\partial B_A$ contains the following facets: 
    	\begin{enumerate}
    		\item $[i, i+1]\cup H\cup \{n-i+1\}$, where $1\leq i\leq \lfloor \frac
    		{n}{2}\rfloor -k+1$ and $H\in \mathcal{F}_{2k-4}^{[i+2, n-i]}$. 
    		\item $[i, i+1]\cup H\cup \{n-i-1\}$, where $1\leq i\leq \lfloor \frac
    		{n}{2}\rfloor -k+1$ and $H\in \mathcal{F}_{2k-4}^{[i+2, n-i-2]}$.
    		\item $\{i+1\}\cup H\cup [n-i, n-i+1]$, where $1\leq i\leq \lfloor \frac
    		{n}{2}\rfloor -k+1$ and $H\in \mathcal{F}_{2k-4}^{[i+2, n-i-1]}$.
    		\item $\{i\}\cup H\cup [n-i-1, n-i]$, where $1\leq i\leq \lfloor \frac
    		{n}{2}\rfloor -k+1$ and $H\in \mathcal{F}_{2k-4}^{[i+2, n-i-2]}$.
    		\item $\{1\} \cup H \cup [n-1, n]$, where $H\in \mathcal{F}_{2k-4}^{[2, n-2]}$.
    		\item $\{\lfloor \frac{n}{2}\rfloor -k+2\}\cup H$, where $H\in \mathcal{F}_{2k-2}^{[\lfloor \frac{n}{2}\rfloor -k+3, \lceil \frac{n}{2}\rceil +k]}$.
    	\end{enumerate}
    \end{lemma}
    \begin{proof}
    	It suffices to show that every item in the list is contained in a unique facet of $B_A$. For $1\leq i\leq \lfloor \frac{n}{2}\rfloor -k+1$, the set $[i, i+1]\cup G$ is a facet of $B_A$ if and only if $$G\leq_p [n-2k+4-i, n-i+1], \quad G\nleq_p [n-2k+2-i, n-i-1];$$ that is, it is a facet of $B_A$ if and only if $G=H\cup [n-i, n-i+1]$, where $H\in \mathcal{F}_{2k-4}^{[i+2, n-i-1]}$, or $G=H\cup [n-i-1, n-i]$, where $H\in  \mathcal{F}_{2k-4}^{[i+2, n-i-2]}$. This shows that the the only facet of $B_A$ that contains the set $F=[i, i+1]\cup H\cup \{n-i+1\}$, where $H\in \mathcal{F}_{2k-4}^{[i+2, n-i]}$, is the facet $F\cup \{m\}$,  where $m$ is the maximum of $[i+2, n-i]\backslash H$. Similarly, the only facet of $B_A$ that contains the second item in the list is $[i, i+1]\cup H\cup [n-i-1, n-i]$.
    	
    	By shifting the indices, a facet of $B_A$ with minimal elements $\{i+1, i+2\}$ must have maximal elements $\{n-i-2, n-i-1\}$ or $\{n-i-1, n-i\}$. Similarly, a facet of $B_A$ with minimal elements $\{i-1, i\}$ must have maximal elements $\{n-i, n-i+1\}$ or $\{n-i+1, n-i+2\}$. This shows that the third and fourth items are also contained in unique facets of $B_A$, namely, in the facets $[i,i+1]\cup H\cup [n-i, n-i+1]$ and $[i,i+1]\cup H\cup [n-i-1, n-i]$, respectively. The fifth and sixth cases are similar: the only facet of $B_A$ that contains $F=\{1\} \cup H \cup [n-1, n]$, where $H\in \mathcal{F}_{2k-4}^{[2, n-2]}$, is $F\cup\{m\}$ where $m$ is the minimum of $[2,n-2]\backslash H$, while the only facet that contains $F'=\{\lfloor \frac{n}{2}\rfloor -k+2\}\cup H$, where $H\in \mathcal{F}_{2k-2}^{[\lfloor \frac{n}{2}\rfloor -k+3, \lceil \frac{n}{2}\rceil +k]}$, is $F'\cup \{\lfloor n/2\rfloor-k+1\}$.
    \end{proof}

    To establish Proposition \ref{prop:mu-of-partialB_S}, we need one additional lemma; its proof follows from the proof of \cite[Proposition 3.5]{BriggsDobbinsLee}.
    \begin{lemma}\label{lm: BDL prop}
    	Any transversal of $\mathcal{F}_{2k}^{[1, n]}$ has size at least $\lceil \frac{n}{2}\rceil -k+1$.
    \end{lemma}
	
    \begin{proposition} \label{prop:mu-of-partialB_S} Let $k\geq 3$. Then 
    	$\liminf_{n\to \infty} \mu(\partial B_{A_n})\geq 2/5$. In particular, $\mu_{2k-1}\geq 2/5$.
    \end{proposition}  
    \begin{proof}
    Fix $n \gg 0$. Let $T$ be a transversal of $\partial B_A$ and let $T^c$ be the complement of $T$ in $[1, n]$. 
		
    Let $i$ be the smallest number in $[1, \lceil  \frac{n}{2}\rceil -k+1]$ such that $T^c$ contains at least one of the following sets: $\{i, i+1, n-i+1\}$, $\{i, i+1, n-i-1\}$, $\{i+1, n-i, n-i+1\}$ and $\{i, n-i-1, n-i\}$. If no such $i$ exists, let $i=\lceil \frac{n}{2}\rceil$. By Lemma \ref{lm: relative squeezed sphere facets} (see the first four items), $T'\coloneqq T \cap [i+2, n-i-2]$ must be a transversal of $\mathcal{F}_{2k-4}^{[i+2, n-i-2]}$. 
    Let $\tilde{T}\coloneqq T \cap ([1, i]\cup [n-i, n])$.  
	By our choice of $i$, $\tilde{T}$ must be a transversal of the following collection $\A_i$ of $3$-sets 	
    	$$\A_i\coloneqq\cup_{j=1}^{i-1} \{ \{j,j+1,n-j+1\}, \{j,j+1,n-j-1\}, \{j+1, n-j, n-j+1\}, \{j, n-j-1,n-j\} \}.$$ 
			
    Since $T \supset \tilde{T}\cup T'$ and since by Lemma  \ref{lm: BDL prop}, $|T'|\geq \frac{1}{2}(n-2i) - O(k)$, to complete the proof, it suffices to show that $|\tilde{T}|\geq \frac{2}{5} \cdot 2i - O(1)$. 
		To do so, write $\tilde{T}^c\cap [1, i]$ as the union of maximal disjoint intervals  $[j_1, j'_1]\cup [j_2, j'_2]\cup \dots \cup [j_m, j'_m]$, where $j'_\ell +1 < j_{\ell+1}$ for all $\ell$, and let $L_\ell=[j_\ell, j'_\ell]$. Using that $\tilde{T}$ is a transversal of $\A_i$, we now find a lower bound on $|\tilde{T}|$ in terms of $\big|\tilde{T}^c\cap [1, i]\big|$, the sizes of these intervals, and $i$. There are the following possible cases for an interval $L_\ell=[j_\ell, j'_\ell]$:
    	\begin{itemize}
    		\item $j'_\ell-j_\ell\geq 2$. Then $R_\ell\coloneqq [n-j'_\ell, n-j_\ell+1]$ is a subset of $\tilde{T}$ and $|R_\ell|=|L_\ell|+1$.
    		\item $j'_\ell-j_\ell=1$. Then $R_\ell\coloneqq\{n-j'_\ell, n-j_\ell+1\}$ is a subset of $\tilde{T}$ and $|R_\ell|=2=|L_\ell|$.
    		\item $j'_\ell=j_\ell$. Then (assuming $1< j_\ell< i$), $\tilde{T}$ contains at least one element in $[n-j_\ell-1, n-j_\ell]$ and one element in $[n-j_\ell+1, n-j_\ell+2]$, and so $R_\ell\coloneqq\tilde{T}\cap [n-j_\ell-1, n-j_\ell+2]$ has size at least $2=|L_\ell|+1$.
    	\end{itemize}
	Since for every $\ell>1$, 	$j_{\ell}\geq j_{\ell-1}'+2$, it follow that the intervals $[n-j_\ell', n-j_\ell+1]$ and $[n-j_{\ell-1}', n-j_{\ell-1}+1]$ are disjoint. Consequently, if both $L_{\ell-1}$ and $L_\ell$ are {\em not singletons} (i.e., $j'_\ell>j_\ell$ and $j'_{\ell-1}>j_{\ell-1}$), then $R_{\ell-1}$ and $R_\ell$ are disjoint. 
		
		What happens if $L_\ell$ (for $\ell>1$) is a singleton? There are two possible cases. If $L_{\ell-1}$ is also a singleton, then by the third bullet point, $R_{\ell-1} \cup R_\ell$ contains at least one element in each of the following pairwise disjoint intervals: $[n-j_\ell-1, n-j_\ell]$, $[n-j_\ell+1, n-j_{\ell-1}]$, and $[n-j_{\ell-1}+1, n-j_{\ell-1}+2]$. Hence, $|R_{\ell-1} \cup R_\ell|\geq 3=(|L_{\ell-1}|+1)+(|L_\ell|+1)-1$.  Similarly, if $L_{\ell-1}$ is not a singleton, then $R_{\ell-1}$ and $R_\ell$ have at most one element in common (namely, $n-j_\ell+2$; this can happen only if  $j_\ell-j'_{\ell-1}=2$), and so in this case, $|R_{\ell-1} \cup R_\ell|\geq |R_{\ell-1}| + |R_\ell|-1$. A similar discussion applies to $R_\ell$ and $R_{\ell+1}$.

	Let $m$ be the total number of intervals. Let $t$ be the number of intervals with $j'-j\geq 2$ (i.e., intervals of size $\geq 3$), $x$ the number of singleton intervals,	$y$ the number of pairs $(L_{\ell-1}, L_{\ell})$ where exactly one of $L_{\ell-1}, L_{\ell}$ is a singleton,  and $z$ the number of pairs $(L_{\ell-1}, L_{\ell})$ where both $L_{\ell-1}$ and $L_{\ell}$ are singletons.  The three bullet points and the above discussion imply that
	\[ \big|\tilde{T}\cap [n-i, n] \big|\geq \left|\bigcup_{\ell=1}^{m}R_\ell\right| \geq \left(\sum_{\ell=1}^{m} \big|L_\ell\big|\right)+t+x-y-z= \big|\tilde{T}^c\cap [1, i]\big|+t+x-y-z.
	\]  
 Adding $\big|\tilde{T}\cap [1, i]\big|$ to the outer parts of this inequality, we conclude that $\big|\tilde{T}\big|\geq i+(t+x-y-z)$. 

It follows from our definitions that $y\leq m-1$ and $x-z\geq y/2$, and so to minimize $t+x-y-z$, we need to maximize $y$ and minimize $t$. This implies that the lower bound on $|\tilde{T}|$ is the smallest when the sequence of intervals $L_\ell$ in $\tilde{T}^c\cap [1, i]$ has the following properties: there are no intervals of size $\geq 3$ (i.e., $t=0$), singletons and $2$-element intervals alternate (i.e., $z=0$ and $y=m-1$), and the gaps between the intervals are as small as possible: every two adjacent intervals are separated by only one element, $j_1\leq 2$, and $j'_m\geq i-2$ (i.e., $x=i/5-O(1)$ and $y=2i/5-O(1)$). Hence $|\tilde{T}| \geq i+(i/5-2i/5)-O(1)=\frac{2}{5} \cdot 2i-O(1)$, as desired. (For example, $\tilde{T}^c\cap [1,i]$ could be $[1,2] \cup\{4\} \cup [6,7] \cup \{9\} \cup \ldots $ in which case $\tilde{T}\cap [1,i]=\{3, 5, 8, 10, \dots\}$.) 
    \end{proof}

    \subsection{The transversal number of $\Delta^d_n$}
 Our goal in this section is to estimate the transversal ratio of the family of cs PL $d$-spheres $\Delta^{d}_n$. We begin by reviewing some relevant definitions, see \cite{NZ-cs-neighborly, NZ-cs-neighborly-new} for more details. If $\Gamma$ and $\Delta$ are pure simplicial complexes of the same dimension and $\Gamma\subset \Delta$, then $\Delta\backslash\Gamma$ denotes the subcomplex of $\Delta$ generated by facets of $\Delta$ that are not facets of $\Gamma$. 

The vertex set of $\Delta^{d}_n$ is $V_n\coloneqq [-1,-n]\cup[1,n]$. The definition of $\Delta^{d}_n$ is by interlaced recursion that introduces not only $\Delta^{d}_n$, but also certain PL $d$-balls $B^{d,i}_n$ that for $i\leq \lceil d/2\rceil -1$ are contained in  $\Delta^{d}_n$. 
    
		\begin{definition} \label{main-constr}
    	Let $d\geq 1$, $i\leq \lceil d/2\rceil$, and $n\geq d+1$ be integers. Define $\Delta^{d}_n$ and $B^{d,i}_n$ inductively as follows: 
	\begin{itemize}
		\item For the initial cases, define
		$\Delta^{1}_{n}$ to be the cycle $(1, 2,\dots, n,-1,-2,\dots, -n,1)$ and $\Delta^{d}_{d+1}$ to be the boundary complex of the $(d+1)$-dimensional cross-polytope $\C^*_{d+1}$; define 
		$B^{d,j}_n\coloneqq\emptyset$   if $j<0$, and $B^{1, 0}_n\coloneqq\overline{\{-1, n\}}$. (In particular, $B^{1,j}_n\subseteq \Delta^{1}_n$ for all $j\leq 0$.)
		
		\item
		If $\Delta^{d-1}_{m}$ and $B^{d-1,i}_m \subseteq \Delta^{d-1}_{m}$ are already defined for all $i\leq \lfloor (d-1)/2\rfloor$ and $m\geq d$, then for $d=2k$ and $n\geq 2k$, define $B^{2k-1,k}_{n}\coloneqq\Delta^{2k-1}_{n}\backslash B^{2k-1,k-1}_{n}$; in addition, for all $n\geq d+1$ and $i\leq \lfloor d/2\rfloor$, define $$B^{d, i}_n\coloneqq\left(B^{d-1, i}_{n-1}*n\right) \cup\left((-B^{d-1,i-1}_{n-1})*(-n)\right).$$ 
		
		\item	 
		If $\Delta^{d}_n$ is already defined, then define $\Delta^{d}_{n+1}$ as the complex obtained from $\Delta^{d}_n$ by replacing the subcomplex $B^{d,\lceil d/2\rceil -1}_n$ with $\partial B^{d,\lceil d/2\rceil -1}_n * (n+1)$ and $-B^{d,\lceil d/2\rceil -1}_n$ with $\partial (-B^{d,\lceil d/2\rceil -1}_n) * (-n-1)$.
	\end{itemize}
\end{definition}
A big portion of \cite{NZ-cs-neighborly} is devoted to proving that $\Delta^d_n$ and $B^{d, i}_n$ are well-defined, that $\Delta^d_n$ is a cs PL sphere while $B^{d, i}_n$ are PL balls, and that $\Delta^d_n$ is cs-$\lceil d/2\rceil$-neighborly.

To estimate the transversal number of $\Delta^d_n$, we need an explicit description of some of the facets of $\Delta^d_n$. For $d=3$ and $d=4$, the complete lists of facets are recorded in \cite[Corollary 3.2]{NZ-cs-neighborly-new} and \cite[Proof of Theorem 5.13]{Pfeifle-20}, respectively. For any $d=2k-1\geq 5$, a large collection of facets of $\Delta^{2k-1}_n$ was identified in \cite[Proposition 3.5]{NZ-cs-neighborly-new}; the following result states a special case of this proposition.
     
We say that a $2k$-set $G=\{p_1, \dots ,p_{2k}\}\subset V_n$ {\em satisfies property $P_n$} if 
    \begin{enumerate}
    	\item $1\leq |p_1|<|p_2|<\dots <|p_{2k}|\leq n$,
    	\item the numbers $p_{2i-1}$ and $p_{2i}$ have the same sign for $1\leq i\leq k$, and
    	\item $|p_{2}|-|p_{1}|=1$ while $|p_{2i}|-|p_{2i-1}|=2$ for $2\leq i\leq k$.
    \end{enumerate}
\begin{lemma} 	\label{rem:facets-of-Delta} 	
    Any $2k$-set $G\subset V_n$ that satisfies property $P_n$ is a facet of $\Delta^{2k-1}_n \backslash \pm B^{2k-1, k-1}_n$. 
\end{lemma}

	The next lemma provides an analogous result for $\Delta^{2k}_n$. 
    \begin{lemma}  \label{Lem:facets-of-Delta}
    	Let $k\geq 2$ and $n<m$. If $F$ is of the form $  \{p_{1},\dots, p_{2k}, p_{2k+1}\}$, where 
    	\begin{enumerate}
    		\item $\{p_{2k-1}, p_{2k}, p_{2k+1}\}=\{n-2, n-1, n+1\}$ or $\{n-2, n, n+1\}$, and
    		\item $G=\{p_1, \dots, p_{2k-2}\}$ has property $P_{n-3}$,
    	\end{enumerate}
    	then $F$ and $-F$ are facets of $\Delta^{2k}_m\backslash \pm B^{2k, k-1}_m$. 
    \end{lemma}
    \begin{proof}
		According to Lemma \ref{rem:facets-of-Delta}, the statement would follow if we show that for every facet $G$ of $\Delta^{2k-3}_{n-3}\backslash {\pm B^{2k-3, k-2}_{n-3}}$, both $F_1\coloneqq G\cup \{n-2, n-1, n+1\}$ and $F_2\coloneqq G\cup \{n-2, n, n+1\}$ are facets of $\Delta^{2k}_m\backslash \pm B^{2k, k-1}_m$. So consider such $G$ and let $i\in\{1,2\}$.
		
    	The maximum of $F_i$ is $n+1$. Hence $F_i$ is a facet of $\Delta^{2k}_m$ if $F_i\in \partial B^{2k, k-1}_{n}*(n+1)\backslash \pm B^{2k, k-1}_{n+1}$. That is, $F_i$ is a facet of $\Delta^{2k}_m$ if $F_i\backslash (n+1) \in \Gamma_n\coloneqq\partial B^{2k, k-1}_{n}\backslash (B^{2k-1, k-1}_{n}\cup B^{2k-1, k-2}_{n})$. By \cite[Section 3, page 7]{NZ-cs-neighborly}, 
			\[B^{2k, k-1}_{n}=
			(B^{2k-2, k-1}_{n-2}*(n-1, n))\cup (-B^{2k-2, k-2}_{n-2}*(n,-n+1, -n))\cup (B^{2k-2, k-3}_{n-2}*(-n, n-1));
		\]
			on the other hand, the facets of $B^{2k-1, k-1}_{n}\cup B^{2k-1, k-2}_{n}$ always contain one of $\pm (n-1)$ and one of $\pm n$. Thus, computing the boundary of $B^{2k, k-1}_{n}$ and keeping in mind that (see \cite[Lemma 3.3]{NZ-cs-neighborly})  
			$B^{2k-2, k-3}_{n-2}\subset -B^{2k-2, k-2}_{n-2} \subset B^{2k-2, k-1}_{n-2}$, we conclude that 
    	\begin{equation*}
    		\begin{split}
    			\Gamma_n&\supseteq \left((B^{2k-2, k-1}_{n-2}\backslash B^{2k-2, k-3}_{n-2})*(n-1)\right) \cup  \left((B^{2k-2, k-1}_{n-2}\backslash (-B^{2k-2, k-2}_{n-2}))*(n)\right) \\
    			&\quad \cup  \left((-B^{2k-2, k-2}_{n-2}\backslash B^{2k-2, k-3}_{n-2})*(-n)\right)=: S_1\cup S_2\cup S_3. \\
    		\end{split}
    	\end{equation*}
			The result follows since if $G$ is a facet of 
			$$\Delta^{2k-3}_{n-3}\backslash \pm B^{2k-3, k-2}_{n-3}=B^{2k-3,k-1}_{n-3} \backslash (-B^{2k-3, k-2}_{n-3}) \subset B^{2k-3,k-1}_{n-3} \backslash B^{2k-3, k-3}_{n-3},$$ then 
			$F_1\backslash (n+1)=G\cup\{n-2,n-1\} \in S_1\subset \Gamma_n$ and $F_2\backslash (n+1)=G\cup\{n-2,n\} \in S_2\subset \Gamma_n$.
    \end{proof}
	
		While the following lemma will not be used in this paper, we state it for completeness.
    \begin{lemma}
		Let $k\geq 2$ and and let $G$ be a $(2k-2)$-set that has property $P_{n-3}$. Then
    $F=G\cup \{n-2, n-1, n\}$ is a facet of $B^{2k, k-1}_n$.
    \end{lemma}
    \begin{proof}
    	This follows immediately from the fact that $B^{2k, k-1}_n\supseteq B^{2k-3, k-1}_{n-3}*\overline{\{n-2, n-1, n\}}$ and $G\in \Delta^{2k-3}_{n-3}\backslash \pm B^{2k-3, k-2}_{n-3} \subset B^{2k-3, k-1}_{n-3}$.
    \end{proof}
		
		We are now in a position to establish the main result of this subsection.
    \begin{proposition} Let $k\geq 2$. Then
    	$\liminf_{n\to \infty} \mu(\Delta^{2k}_n)\geq \frac{2}{5}$ while $\liminf_{n\to \infty} \mu(\Delta^{2k-1}_n)\geq \frac{1}{2}$. In particular, $\mu_{2k+1}\geq 2/5$ and $\mu_{2k}\geq 1/2$.
    \end{proposition}
    \begin{proof}
		We first discuss the case of $d=2k$.
    	Let $T=T_+\cup T_-$ be a transversal of $\Delta^{d}_n$, where $T_+$ and $T_-$ consist of positive and negative vertices, respectively. It suffices to show that $|T^+|\geq \frac{2}{5}n +O(k)$. (By symmetry, this will also imply that $|T^-|\geq \frac{2}{5}n +O(k)$.) To do so, write $T_+= \cup_{p=1}^m I_p$ as the union of maximal disjoint intervals. Similarly, write the complement of $T_+$ in $[1,n]$ as the union of maximal disjoint intervals $T_+^c=\cup_{p=1}^{m'} J_p$. Let $[a_1, b_1]$ be the first non-singleton interval in $T^c_+$. Then $|T^+\cap [1,a_1-1]|\geq \frac{1}{2}(a_1-1)$. Hence to complete the proof, it suffices to show that $|T_+\cap[a_1,n]|\geq \frac{2}{5}(n-a_1)-O(k)$. This will follow from the next two claims.
			
\smallskip\noindent{\bf Claim 1}: {\em The intervals in $T^c_+$ of size $\geq 4$ together occupy at most $O(k)$ elements.} \\
\proof Let $[a_\ell, b_\ell]$ be the last interval of size $\geq 4$ in $T^c_+$, and let $[a_2, b_2]<\dots<[a_{\ell-1}, b_{\ell-1}]$ be the intervals of size $\geq 3$ in $T_+^c$ that lie between $[a_1,b_1]$ and $[a_\ell, b_\ell]$. Define the following sets:
		\begin{eqnarray*}
			A_1 &=& \{a_1,a_1+1\} \cup \bigcup_{i=1}^{\lfloor \frac{b_1-a_1-1}{3}\rfloor} \{a_1+3i-1, a_1+3i+1\} =\{a_1,a_1+1, a_1+2,a_1+4,\dots\} \subset [a_1,b_1],\\
			A_j &=& \bigcup_{i=0}^{\lfloor \frac{b_j-a_j-2}{3}\rfloor} \{a_j+3i, a_j+3i+2\} = \{a_j,a_j+2,a_j+3,a_j+5,\dots\}\subset [a_j,b_j],  \,\,\, 2\leq j \leq \ell-1,\\
			A_\ell&=& \Big(\bigcup_{i=2}^{\lfloor \frac{b_\ell-a_\ell}{3}\rfloor}\{b_\ell-3i, b_\ell-3i+2\}\Big) \cup \{b_\ell-3, b_\ell-2, b_\ell\} =\{\dots, b_\ell-3, b_\ell-2, b_\ell\} \subset [a_\ell,b_\ell]. 
			\end{eqnarray*} 
	Note that $\sum_{j=1}^\ell|A_j|<2k$ as otherwise, by Lemma \ref{Lem:facets-of-Delta}, the set $\cup_{j=1}^\ell A_j \subset T^c_+$ would contain a facet of $\Delta^{2k}_n$. But such a facet	would be disjoint from $T$ contradicting our assumption that $T$ is a transversal of $\Delta^{2k}_n$.	Thus $\ell\leq k-1$ and $$\left\lfloor \frac{b_1-a_1+2}{3}\right\rfloor +\left\lfloor \frac{b_\ell-a_\ell}{3}\right\rfloor+\sum_{j=2}^{\ell-1} \left\lfloor \frac{b_j-a_j+1}{3}\right\rfloor\leq k-1.$$ We conclude that $\sum_{j=1}^\ell \big|[a_j,b_j]\big|= \ell+\sum_{j=1}^\ell(b_j-a_j)=O(k)$, as claimed.
			
	\smallskip\noindent{\bf Claim 2}:	{\em The number of singleton intervals in $T^+\cap [b_1+1,n]$ that are adjacent on the right to an interval in $T^c_+$ of size $\geq 2$ is at most $k-1$.}	\\
	\proof  
		If there were at least $k$ such intervals, say, $\{y_1\}, \dots, \{y_k\}$, then, by Lemma \ref{Lem:facets-of-Delta}, $$F\coloneqq\{a_1,a_1+1\} \cup \big(\bigcup_{i=2}^{k-1} \{ y_i-1,y_i+1\}\big) \cup \{y_k-1, y_k+1,y_k+2\}$$ would be a facet of  $\Delta^{2k}_n$ that is disjoint from $T$. This is impossible since $T$ is a transversal of $\Delta^{2k}_n$.

			To finish the proof of the proposition, observe that the above two claims imply that there is a set $C$ of size $O(k)$ such that $[a_1,n]\backslash C$ is the disjoint union of pairs of adjacent intervals $(I,J)$ where $I\subset T_+$, $J\subset T^c_+$ and such that $|I|\geq 2, |J|\leq 3$ or $|I|=|J|=1$. In particular, $T_+\cap [a_1,n]\geq \frac{2}{5}(n-a_1)-O(k)$.
			
			The case of $d=2k+1$ is very similar: the main difference is that in this case, using Lemma \ref{rem:facets-of-Delta} and taking $[a_\ell, b_\ell]$ to be the last interval of size $\geq 3$, allows to strengthen Claim 1 to the statement that {\em the intervals in $T^c_+$ of size $\geq 3$ together occupy at most $O(k)$ elements.} 
    \end{proof}
		
		We would like to point out that $\lim_{n\to\infty} \mu(\Delta^3_n) =1/2$ while $\lim_{n\to\infty} \mu(\Delta^4_n)=2/5$. Indeed, by looking at the complete list of facets of $\Delta^3_n$ and $\Delta^4_n$ (see \cite[Corollary 3.2]{NZ-cs-neighborly-new} and \cite[Proof of Theorem 5.13]{Pfeifle-20}, respectively), one can easily check that $T=(\pm\{1,3,5,7,\dots\}\cap V_n)\cup\{\pm n\}$ is a transversal of $H(\Delta^3_n)$; similarly, $T=(\pm\{1,2,6,7,11,12,\dots\}\cap V_n) \cup \{\pm n\}$ is a transversal of $H(\Delta^4_n)$.
    
		To summarize the discussion of this section, the following lower bounds on $\mu_d$ are known at present: $\mu_3=1/2$,  $\mu_4\geq 11/21$,  and for $k\geq 3$, $\mu_{2k}\geq 1/2$ while $\mu_{2k-1}\geq 2/5$. For $d>3$, no upper bounds on $\mu_d$, except for the trivial one $\mu_d\leq 1$, are known. Shedding any new light on possible values of $\mu_d$ would be of great interest; so will be shedding any light on possible values of $\lambda_d$ from Problem \ref{prob:flag}. At the moment nothing at all is known about these numbers.

\section*{Acknowledgments} We are grateful to Rowan Rowlands for his comments on the preliminary version of this paper.
		
	{\small
		\bibliography{refs}
		\bibliographystyle{plain}
	}
\end{document}